\documentclass{amsart}

\newtheorem{theorem}{Theorem}[section]
\newtheorem{lemma}[theorem]{Lemma}
\newtheorem{proposition}[theorem]{Proposition}

\theoremstyle{definition}
\newtheorem{definition}[theorem]{Definition}

\numberwithin{equation}{section}

\newcommand{\bes}{\begin{equation*}}
\newcommand{\beq}{\begin{equation}}
\newcommand{\ees}{\end{equation*}}
\newcommand{\eeq}{\end{equation}}

\newcommand{\T}{\mathbb{T}}
\newcommand{\D}{\mathbb{D}}
\newcommand{\N}{\mathbb{N}}
\newcommand{\Rp}{\text{Re }}

\newcommand{\ns}[1]{#1^{\star}}
\newcommand{\vs}[1]{#1^{\diamond}}
\newcommand{\conj}[1]{\overline{#1}}
\DeclareMathOperator{\sgn}{sgn}

\begin{document}

\title{Continuity of Extremal Elements in Uniformly Convex Spaces}
\author{Timothy Ferguson}
\address{
Department of Mathematics\\1326 Stevenson Center\\Vanderbilt University\\Nashville, TN 37240}
\email{timothy.j.ferguson@vanderbilt.edu}

\date{\today}

\begin{abstract}
In this paper, we study the problem of finding the 
extremal element for a linear functional over a uniformly convex 
Banach space.  We show that a unique extremal element exists
and depends continuously on the linear functional, and vice versa. 
Using this, we simplify and clarify 
Ryabykh's proof that for any linear functional on a uniformly convex 
Bergman space with kernel in a certain Hardy space, 
the extremal functional belongs to the corresponding Hardy space.
\end{abstract}

\maketitle

Our purpose is to study extremal problems over uniformly convex 
Banach spaces. 
For a given linear functional $\phi$, what elements 
$x$ of the space with norm $\|x\|=1$ maximize 
$\Rp \phi(x)$?  Because of the uniform convexity, this problem will always 
have a unique solution, which is called the extremal element. 
In this paper, we show that the extremal element 
depends continuously on the functional $\phi$, and it 
can be approximated by the solutions of the same extremal problem over 
subspaces of the original Banach space.  
Moreover, we show that an extremal element can arise from at most one 
linear functional of unit norm, and that the functional depends 
continuously on the extremal element. 
Using these results, we 
give a streamlined proof of a theorem of Ryabykh, which says 
that for a functional defined on the Bergman space $A^p$, with 
kernel in the Hardy space $H^q$, the extremal element is in $H^p,$ where 
$1 < p < \infty$ and $\tfrac{1}{p} + \tfrac{1}{q} = 1.$

\section{Uniform convexity and extremal problems}

Let $X$ be a complex Banach space and let 
$X^*$ be its dual space. 
For a given linear functional 
$\phi \in X^*$ with $\phi \ne 0,$ 
we are interested in all elements $x \in X$ with norm 
$\|x\| = 1$ such that 
\begin{equation}\label{norm1}
\Rp \phi(x) = \sup_{\|y\|=1} \Rp \phi(y) = \| \phi \|.
\end{equation}
Such a problem is referred to as an extremal problem, and $x$ is an extremal 
element. 

This problem has been studied extensively for certain Banach spaces.  
For example, for the Hardy spaces $H^p$ where $1\le p \le \infty$, 
the problem was investigated independently around 1950 
by S. Ya. Khavinson, and by Rogosinski and Shapiro.  
(See Duren \cite{D_Hp}, Chapter 8.) 
The problem for the Bergman spaces $A^p$ for $1 \le p < \infty$ has been 
studied as well, for example by Vukoti\'{c} \cite{Dragan} and Khavinson and 
Stessin \cite{Khavinson_Stessin}.  We will define the Hardy and 
Bergman spaces in Section \ref{section_hp_ap}. 

A closely related problem is that of finding $x\in X$
such that 
\begin{equation}\label{value1}
\phi(x) = 1 \qquad \text{  and } \qquad 
\|x\| = \inf_{\phi(y) = 1} \|y\|.
\end{equation}
If $x$ solves the problem \eqref{norm1}, then $\frac{x}{\phi(x)}$ 
solves the problem 
\eqref{value1}, and if $x$ solves \eqref{value1}, then 
$\frac{x}{\|x\|}$ solves \eqref{norm1}. To standardize notation, we 
will often denote solutions to \eqref{norm1} by 
$\ns{x}$ and solutions to \eqref{value1} by $\vs{x}.$  

The problem \eqref{norm1} need not have a solution, and if it does the 
solution need not be unique.  However, if the Banach space 
$X$ is {\it uniformly convex}, there will always be a unique solution.

\begin{definition} A Banach space $X$ is said to be uniformly convex if 
for each $\varepsilon > 0$, there is a $\delta > 0$ such that for all 
$x,y \in X$ with $\|x\| = \|y\| = 1,$ 
$$\left\| \tfrac12(x+y) \right\| > 1 - \delta \qquad \text{ implies } 
\qquad 
\|x-y\| < \varepsilon.
$$
\end{definition}
An equivalent statement is that if $\{x_n\}$ and $\{y_n\}$ are 
sequences in $X$ such that $\|x_n\| = \|y_n\| = 1$ and 
$\|x_n + y_n\| \rightarrow 2,$ then $\|x_n - y_n\| \rightarrow 0.$
Uniform convexity was introduced by Clarkson \cite{Clarkson}, 
who proved that the $L^p$ spaces are uniformly convex for 
$1<p<\infty$.

\begin{proposition}\label{uc_sequence}
Let $X$ be a uniformly convex Banach space, and let 
$\{x_n\}$ and $\{y_n\}$ be sequences in $X.$  If for some 
$d > 0,$ $\|x_n\| \rightarrow d$ and $\|y_n\| \rightarrow d$ as 
$n \rightarrow \infty,$ and $\|x_n + y_n\| \rightarrow 2d,$ then 
$\|x_n - y_n\| \rightarrow 0.$ 
\end{proposition}
\begin{proof}
We have that 
\bes
\begin{split}
2 &\ge \bigg \| \frac{x_n}{\|x_n\|} + \frac{y_n}{\|y_n\|} \bigg\| 
= \frac{1}{\|x_n\|} \bigg\| x_n + \frac{\|x_n\|}{\|y_n\|} y_n \bigg\| \\
&\ge \frac{1}{\|x_n\|}\|x_n + y_n\| - 
    \frac{1}{\|x_n\|}\bigg\|y_n - \frac{\|x_n\|}{\|y_n\|} y_n \bigg\| 
\rightarrow 2
\end{split}
\ees
as $n \rightarrow \infty.$ Hence by uniform convexity, 
$$ \bigg\|\frac{x_n}{\|x_n\|} - \frac{y_n}{\|y_n\|}\bigg\| \rightarrow 0.$$
But then
\bes 
\begin{split}
\|x_n - y_n\| &= \|x_n\|\ \bigg\|\frac{x_n}{\|x_n\|} - \frac{y_n}{\|x_n\|}
\bigg\| \\
&\le\|x_n\|\ \bigg\|\frac{x_n}{\|x_n\|} - \frac{y_n}{\|y_n\|}\bigg\| + 
\|x_n\| \left\|y_n\left(\frac{1}{\|x_n\|} - \frac{1}{\|y_n\|}\right)\right\| 
\rightarrow 0.
\end{split}
\ees
\end{proof}

The following result is basic (see for instance \cite{D_Ap}, Section 2.2) 
and gives immediately the existence and uniqueness 
of extremal elements.
\begin{proposition}\label{cl_conv} 
A closed convex subset of a uniformly convex 
Banach space has exactly one element of smallest norm.
\end{proposition}

Since
the problem \eqref{value1} is one of finding an element of minimal norm 
over a (closed) subspace of the Banach space, 
it has a unique solution, and thus 
we obtain the following theorem.
\begin{theorem}\label{unique_solution}
If $X$ is a uniformly convex Banach space and $\phi \in X^*$ with 
$\phi \ne 0,$ then the problems \eqref{norm1} and \eqref{value1} both 
have a unique solution.
\end{theorem}

For later reference, we record the relations
\begin{equation}\label{eq_basics}
\begin{split}
\ns{x} &= \frac{\vs{x}}{ \|\vs{x}\|} = \| \phi \| \vs{x}, \\
\vs{x} &= \frac{\ns{x}}{\phi(\ns{x})} = \frac{\ns{x}}{\| \phi \|}, \\
\| \phi \| &= \phi(\ns{x}) = \frac{1}{\| \vs{x} \|}.\\
\end{split}
\end{equation}

We now state a lemma which will be applied repeatedly. 
\begin{lemma}\label{uc_sequence2}
Let $X$ be a uniformly convex Banach space, let 
$\{x_n\}$ and $\{y_n\}$ be sequences in $X$, and let 
$\phi \in X^*$ where $\phi \ne 0.$  If for some 
$d \ge 0,$ $\|x_n\| \rightarrow d$ and $\|y_n\| \rightarrow d$ as 
$n \rightarrow \infty,$ and if $|\phi(x_n + y_n)| \rightarrow 2d\|\phi\|$, 
then $\|x_n-y_n\| \rightarrow 0.$
\end{lemma}
\begin{proof}
Since $|\phi(x_n + y_n)| \le \|\phi\| \|x_n + y_n\|$,
we have that
$$ \frac{|\phi(x_n + y_n)|}{\|\phi\|} \le \|x_n + y_n\| \le 
\|x_n\| + \|y_n\|.$$
But the left and right sides of this inequality both approach $2d$, 
Proposition \ref{uc_sequence} gives the result.
\end{proof} 

\section{$H^p$ and $A^p$ spaces}\label{section_hp_ap}
We now recall some basic facts about Hardy and Bergman spaces.  For proofs 
and further information, see \cite{D_Hp} and \cite{D_Ap}. 
Suppose that $f$ is analytic in the unit disc.  For $0 < p < \infty$ and 
$0 < r < 1,$ the integral mean is 
$$M_p(f,r)=
\bigg\{ \frac{1}{2\pi} \int_{0}^{2\pi} |f(re^{i\theta})|^p d\theta 
\bigg\}^{1/p}.$$
If $p=\infty$, we write
$$
M_\infty(f,r) = \sup_{\theta} |f(re^{i\theta})|.$$
For fixed $f$ and $p$, the integral means are increasing 
functions of $r.$  If $\sup_r M_p(f,r) < \infty,$ we say that $f$ is in 
the Hardy space $H^p.$  
For any function $f$ in $H^p,$ the radial limit 
$f(e^{i\theta}) = \lim_{r \rightarrow 1^-} f(re^{i\theta})$ 
exists for almost every $\theta.$   
 An $H^p$ function is uniquely determined by 
the values of its limit function on any set of positive measure.  
$H^p$ is a Banach space with norm 
$$\| f \|_{H^p} = \sup_r M_p(f,r) = \|f(e^{i\theta})\|_{L^p}.$$
Because of this and the fact that functions in $H^p$ are determined by 
their boundary values, we can consider $H^p$ as a subspace of 
$L^p(\T),$ where $\T$ denotes the unit circle.  
Thus, $H^p$ is uniformly convex for $1 < p < \infty$ and its dual space  
is isometrically isomorphic to $L^q/H^q,$ where $1/p + 1/q = 1$, with each 
equivalence class $[g] \in L^q/H^q$ representing the functional 
$$\phi(f) = \int_{\T} f(z) g(z) \, dz.$$

The Bergman space $A^p$ for $0 < p <\infty$ consists of all functions $f$
analytic in the unit disc $\D$ with 
$$ \| f\|_{A^p}^p = \int_{\D} |f(z)|^p\,d\sigma(z) < \infty,$$
where $\sigma$ is Lebesgue measure divided by $\pi.$ 
Because $A^p$ is a closed subspace 
of $L^p(\D)$, it is uniformly convex for $1 < p < \infty.$   
For $1 < p < \infty,$ the dual of $A^p$ is isomorphic to $A^q$, 
where $1/p + 1/q = 1,$ and the element $g\in A^q$ represents the functional 
defined by $\phi(f) = \int_{\D} f(z)\conj{g(z)}\,d\sigma(z)$. 
This isomorphism is not an isometry, but  
if the functional $\phi$ is represented by the 
function $g \in A^q$, then
\begin{equation}\label{A_q_isomorphism}
\| \phi \| \le \| g \| \le C_p \| \phi \|
\end{equation}
where $C_p$ is a constant depending only on $p$. 
We remark that $H^p$ is contained in $A^p$, and in fact 
$\|f\|_{A^p} \le \|f\|_{H^p}.$ 

The following theorem is essentially known 
(\textit{cf.} \cite{Shapiro_Approx}).
\begin{theorem}\label{integral_extremal_condition} 
Let $X$ be a (closed) subspace of $ A^p$ for $1 < p < \infty.$  Then 
a function $f \in X$ with $\|f\| = 1$ solves the extremal problem 
 \eqref{norm1} if and only if  
$$\int_{\D} h |f|^{p-1} \conj{\sgn f} \, d\sigma = 0,$$
for all $h \in X$ with $\phi(h) = 0.$  
More generally, for all $h\in X$,
$$\int_{\D} h |f|^{p-1} \conj{\sgn f}  \,  d\sigma 
= \frac{\phi(h)}{\| \phi \|}.$$
\end{theorem}

Extremal problems for $H^p$ are well understood, due to the successful 
use of ``dual extremal problems'' (see \cite{D_Hp}, Chapter 8), 
but the method does not work as well for $A^p$
(see \cite{Dragan}).  

\section{Continuous dependence of the solution on the functional}

It is important to know whether the extremal element depends continuously 
on the functional.  This turns out to be true for uniformly convex spaces.  
\begin{theorem}\label{solution_continuity}
Suppose that $X$ is a uniformly convex Banach space and that 
$\{\phi_n\}$ is a sequence of nonzero functionals in $X^*$ such that 
$\phi_n \rightarrow \phi \ne 0.$  Let $\ns{x_n}$ denote the solution 
to problem \eqref{norm1} for $\phi_n,$ and let $\ns{x}$ be the solution 
for $\phi.$  Similarly, let $\vs{x_n}$ denote the solution 
to problem \eqref{value1} for $\phi_n,$ and let $\vs{x}$ be the solution 
for $\phi.$  Then $\ns{x_n} \rightarrow \ns{x}$ and 
$\vs{x_n} \rightarrow \vs{x}.$  
\end{theorem}
 Ryabykh \cite{Ryabykh} gives a different proof of this 
statement while establishing another theorem, 
but our proof is simpler and more direct.  
\begin{proof}
Note that 
$$ \phi(\ns{x_n}) = \phi_n(\ns{x_n}) + (\phi - \phi_n)(\ns{x_n}) = 
\|\phi_n\| + (\phi - \phi_n)(\ns{x_n}) \rightarrow \|\phi\|$$
so that 
$$\phi(\ns{x_n} + \ns{x}) \rightarrow 2 \|\phi\|.$$ \
Lemma \ref{uc_sequence2} now shows that $\ns{x_n} \rightarrow \ns{x}$.  
It follows that $\vs{x_n} \rightarrow \vs{x}$ since 
$$\vs{x_n} = \frac{\ns{x_n}}{\| \phi_n \|} \rightarrow 
\frac{\ns{x}}{\| \phi \|} = \vs{x}.$$
\end{proof}

For given $\phi,$ a unique $\vs{x}$ solves the problem 
\eqref{value1}:
$$ \| \vs{x} \| = \min_{\phi(x) = 1} \|x\|.$$
It is natural to ask whether different functionals can give rise to the same 
solution of the problem \eqref{value1}.  
(Clearly, any two functionals that are positive real   
multiples of each other will have the same extremal element solving 
\eqref{norm1}).  
The following theorem answers 
this question when $X^*$ is uniformly convex. 

\begin{theorem}\label{dual_uc}
Let $X$ be a Banach space and let $x \in X$ with $x\ne 0.$  
If $X^*$ is uniformly convex, 
then there exists a unique $\phi \in X^*$ such that x solves the problem 
\eqref{value1} associated with $\phi.$  
\end{theorem}
\begin{proof}
By the Hahn-Banach theorem, there is some $\phi \in X^*$ such that 
$\phi(x) = 1$ and $\| \phi \| = \frac{1}{\|x\|}.$  But if 
for some $y\in X,$ $\phi(y) = 1,$ 
then $1 \le \| \phi \| \| y \|,$ or
$\|y\| \ge \| \phi \|^{-1} = \|x\|$. This says that 
$x$ solves the problem \eqref{value1} associated with $\phi.$ 
To show that $\phi$ is unique, 
consider the problem
of finding $\vs{\psi}$ such that 
\begin{equation} \label{eq_dualproblem}
 \| \vs{\psi} \| = \min_{\psi \in X^*, \psi(x)=1} \| \psi \|
\end{equation}
We claim that if $x$ solves the problem \eqref{value1} for some 
$\theta \in X^*,$ then $\theta$ solves the problem \eqref{eq_dualproblem}. 
In particular, $\phi$ solves the problem \eqref{eq_dualproblem}. 
To see this, note that if $x$ solves \eqref{value1} for $\theta,$ then  
$\theta(x) = 1.$  If $\theta$ is not a solution of \eqref{eq_dualproblem}, 
then there is a 
functional $\psi$ such that $\| \psi \| < \| \theta \|$ and 
$\psi(x) = 1.$  But this is impossible, since it would imply  
$$1 = |\psi(x)| \le \| \psi \| \| x \| 
= \frac{\| \psi \|}{\| \theta \|} < 1,$$
where we have used last relation in \eqref{eq_basics}.   
Since $X^*$ is uniformly convex, 
Theorem \ref{unique_solution} shows that 
$\phi$ is the unique solution to \eqref{eq_dualproblem}, 
which proves the theorem.
\end{proof}

When $\vs{x}$ determines the functional $\phi$ 
uniquely, it is also natural to ask whether $\phi$ depends continuously 
on $\vs{x}.$ The following theorem answers this question when 
$X^*$ is uniformly convex.

\begin{theorem}\label{functional_continuity}
(a) Suppose that $X$ is a Banach space whose dual space $X^*$ is uniformly
 convex.  If $S$ is a closed subspace of $X$, then for any $x \in S$, 
there exists a unique $\phi \in S^*$ such that x solves the problem 
\eqref{value1} associated with $\phi$ over $S$.  

(b) Moreover, if $x_n \in S$ and $x_n \rightarrow x,$ and 
$\phi_n$ is the unique functional in $S^*$ that solves the 
problem \eqref{value1} for $x_n,$ then $\phi_n \rightarrow \phi.$ 
\end{theorem}

\begin{proof}
Recall that if $S$ is a closed subspace of $X,$ then 
$S^*$ is isometrically isomorphic to $X^*/S^\perp$, where $S^\perp$ is the annihilator of 
$S$ in $X^*$.  In \cite{Kothe}, Section 26, it is shown that 
the quotient space of a uniformly convex space is uniformly convex, which 
shows that $S^*$ is uniformly convex. 
From this and Theorem \ref{dual_uc}, part (a) follows.

Since, as shown in the proof of Theorem 
\ref{dual_uc}, each $\phi_n$ is the unique solution to the problem 
\eqref{eq_dualproblem} with $x_n$ in place of $x$, and since 
$\phi$ is the unique solution of the problem \eqref{eq_dualproblem}, 
Theorem \ref{solution_continuity} implies part (b).
   
\end{proof}

Since $(L^p)^* = L^q$ is uniformly convex for $1 < p < \infty$, 
this theorem applies to the spaces $A^p$ and $H^p$ for 
$1 < p < \infty.$ 

\section{Approximation by solutions in subspaces and Ryabykh's theorem}

The preceding results will now be applied to give a streamlined 
proof of a theorem of Ryabykh \cite{Ryabykh}, as indicated in the 
introduction.  For this purpose it will be helpful to apply the 
following theorem, which allows an extremal element to be approximated 
by extremal elements over subspaces. 

\begin{theorem}\label{subspace_approx}
Suppose that $X$ is a uniformly convex Banach space and let 
$X_1$, $ X_2$, $\cdots$ be (closed) subspaces for which  
$X_1 \subset X_2 \subset \cdots \subset X$ and 
$$\overline{\bigcup_{n \in \N} X_n} = X.$$
Let $\phi \in X^*$, and let 
$$\| \phi \|_n= \sup_{x \in X_n, \|x\|=1} |\phi(x)|.$$  Let 
$\ns{x}_n$ denote the solution to the problem \eqref{norm1} when restricted 
to the subspace $X_n,$ and let $\vs{x}_n$ denote the 
solution to the problem \eqref{value1} when restricted to $X_n.$  
Then $\| \phi \|_n \rightarrow \| \phi \|,$ 
$\ns{x}_n \rightarrow \ns{x}$, and $\vs{x}_n \rightarrow \vs{x}$ as 
$n\rightarrow \infty.$
\end{theorem}
Here, $\ns{x}$ denotes the solution to \eqref{norm1} over  $X$, 
and $\vs{x}$ denotes the solution to \eqref{value1} over $X$.
\begin{proof}
First of all, we know that each $\ns{x}_n$ and $\vs{x}_n$ is uniquely 
determined since 
a closed subspace of a uniformly convex space is uniformly convex.  Let 
$\varepsilon > 0$ be given.  Since $\cup_{n\in N} X_n$ is dense in $X$, 
we may choose an $n$ such that 
$\| \ns{x} - y \| < \varepsilon$ for some $y \in X_n.$  
Thus 
\bes \begin{split} |\phi(y)| &= | \phi(\ns{x}) - \phi(\ns{x} - y)| 
\ge | \phi(\ns{x}) | - | \phi(\ns{x} - y) | 
= \| \phi \| - | \phi(\ns{x} - y) | \\
&\ge \|\phi \| - \| \phi \|\, \|\ns{x} - y \| \ge \| \phi \|(1 - \varepsilon).
\end{split} \ees
We also know that $\|y \| \le 1 + \varepsilon,$ so
$$\| \phi \|_n \ge \frac{|\phi(y)|}{\|y\|} \ge 
\frac{(1 - \varepsilon)\| \phi \|}{1 + \varepsilon},$$ 
and thus for all $N \ge n,$ $$ \| \phi \|_N \ge 
\frac{(1 - \varepsilon)\| \phi \|}{1 + \varepsilon}.$$ 
But since $\| \phi \| \ge \| \phi \|_m$ for all $m,$ this implies that 
$$ \| \phi\| \ge \limsup_{m \rightarrow \infty} \| \phi \|_m 
 \ge \liminf_{m \rightarrow \infty} \| \phi \|_m \ge 
\frac{(1 - \varepsilon)\| \phi \|}{1 + \varepsilon}.$$
Because 
$\varepsilon$ was arbitrary,  this shows that 
$\| \phi \|_m \rightarrow \|\phi\|.$ 

Now, $\phi(\ns{x_n} + \ns{x}) = \| \phi \|_n + \| \phi \| \rightarrow 
2 \| \phi \|,$ so Lemma \ref{uc_sequence2} 
shows that $\|\ns{x}-\ns{x}_n\|\rightarrow 0$. 
For $\vs{x},$ the result now follows since 
$$ \vs{x}_n = \frac{\ns{x}}{\| \phi \|_n}\qquad \text{ and } \qquad   
 \vs{x} = \frac{\ns{x}}{\| \phi \|}.$$
\end{proof}

With the help of the preceding results, we can now obtain a slightly sharpened version 
of Ryabykh's theorem.  Our proof adapts some of Ryabykh's ideas but is simpler 
and more concise.  
 We note that Ryabykh's approach successfully applies to other extremal 
problems as well.  For example, it applies to some problems of best approximation of a given function by harmonic or analytic functions 
(see \cite{Khavinson_McCarthy_Shapiro}), and to some 
nonlinear extremal problems involving nonvanishing functions 
(see \cite{Khavinson_nonvanishing}).  
\begin{theorem}\label{ryabykh_thm} 
Let $1 < p < \infty$ and let $1/p + 1/q = 1.$  Suppose that 
$\phi \in (A^p)^*$ and $\phi(f) = \int_{\D} f \conj{g} \, d\sigma$ 
for some $g \in H^q,$ where $g \ne 0.$  
Then the solution to the extremal problem \eqref{norm1} (with  
$X = A^p$) belongs to 
$H^p$ and satisfies 
\begin{equation}\label{ryabkh_estimate}
\|\ns{f}\|_{H^p} \le \Bigg\{ \bigg[ \max(p-1,1)
\bigg]\frac{C_p\|g\|_{H^q}}{\| g \|_{A^q}}\Bigg\}^{q/p},
\end{equation}
where $C_p$ is the constant in \eqref{A_q_isomorphism}.
\end{theorem}
Of course, this implies that the solution to the problem \eqref{value1} is 
in $H^p$ as well.
Note that the constant $C_p \rightarrow \infty$ as 
$p \rightarrow 1$ or $p \rightarrow \infty$.

\begin{proof}
Let 
$$ G(z) = \frac{1}{z} \int_0^{z} g(\zeta) \, d\zeta,$$
so that $(zG)' = g.$  

Now, the continuous form of Minkowski's inequality gives
\bes
\begin{split} 
M_q(zG,r) &=  \left \{ \frac{1}{2\pi} \int_{0}^{2\pi} |re^{i\theta}G(e^{i\theta})|^q d\theta 
\right\}^{1/q} = 
 \left \{ \frac{1}{2\pi} \int_{0}^{2\pi} \left|\int_0^r g(\rho e^{i\theta}) e^{i\theta} d\rho \right|^q  d\theta
\right\}^{1/q} \\ &\le 
 \int_{0}^{r} \left\{ \frac{1}{2\pi}\int_0^{2\pi} 
|g(\rho e^{i\theta})|^q d\theta  \right\}^{1/q} d\rho 
\le \int_0^r M_q(g,\rho) d\rho \le \|g\|_{H^q}.
\end{split}
\ees
Thus, 
\begin{equation} \label{Omega_estimate}
|G\|_{H^q} \le \|g\|_{H^q}.
\end{equation}

We will need the Cauchy-Green theorem, 
which will allow us to relate the $H^p$ and $A^p$ spaces. 
For any $f \in C^1(\overline{\D}),$  it states that 
$$ \frac{1}{2 i} \int_{\T} f(z) \, dz = \pi \int_{\D} 
\frac{\partial}{\partial \conj{z}} f(z) \,d\sigma(z).$$

To facilitate calculations involving the Cauchy-Green theorem, we suppose 
first that $g\in C^1(\overline{\D}).$  
Let $\ns{f}_n$ denote the solution to the extremal problem 
\eqref{norm1} over the space of all polynomials of degree $n$ or less, 
considered as a subspace of $A^p,$ and let $\ns{f}$ be the solution to the 
same problem over the space $A^p.$ 
Then by the Cauchy-Green theorem,
\bes
\begin{split}
\|\ns{f}_n\|_{H^p}^p 
&= \frac{1}{2\pi} \int_0^{2\pi} |\ns{f}_n(e^{i\theta})|^p\, d\theta = 
\frac{1}{2\pi i} \int_{\T} |\ns{f}_n(z)|^p\, \conj{z}\,dz \\ 
\\ &=
 \int_{\D}  
\left(\ns{f}_n + \frac{p}{2}z{\ns{f}_n}'\right) 
|\ns{f}_n|^{p-1} \conj{\sgn\ns{f}_n} \,d \sigma. 
\end{split}
\ees
Because $(\ns{f}_n + \frac{p}{2}z{\ns{f}_n}')$ is a polynomial 
of degree at most $n,$ we can appeal to 
Theorem \ref{integral_extremal_condition} with $X$ taken to be the subspace 
of $A^p$ consisting of all such polynomials.  The theorem shows that  
\bes
\begin{split}
\| \ns{f}_n\|_{H^p}^p &= 
\frac{1}{\| \phi \|_n} \,\phi\left( \ns{f}_n + \frac{p}{2} z {\ns{f}_n} ' \right) = 
\frac{1}{\|\phi\|_n}
\int_{\D} \conj{g} \left(\ns{f}_n + \frac{p}{2}z{\ns{f}_n}'\right)\, 
d \sigma \\ 
&= \frac{1}{\|\phi\|_n} 
\int_{\D}\left[ \frac{\partial}{\partial \conj{z}}\left( \ns{f}_n \conj{z G} \right) + 
\frac{p}{2}
\left(\frac{\partial}{\partial z}\left(z \ns{f}_n \conj{g}\right) - 
\frac{\partial}{\partial \conj{z}} \left(\ns{f}_n \conj{zG}\right)\right)\right] \,d\sigma.
\end{split} \ees
Now another application of the Cauchy-Green theorem gives:
\bes \begin{split}
\|\ns{f}_n\|_{H^p}^p
&= \frac{1}{2\pi i \|\phi\|_n } \int_{\T}  \conj{z G} \ns{f}_n\, dz  
- \frac{p}{4 \pi i \|\phi\|_n} \int_{\T} z \ns{f}_n \conj{g}\, d\conj{z} - 
\frac{p}{4 \pi i \|\phi\|_n } \int_{\T}  \ns{f}_n \conj{z G}\, dz\\
&= 
\frac{1}{2\pi \|\phi\|_n} \int_{0}^{2 \pi}
\ns{f}_n \left[\left(\frac{p}{2}\right)\conj{g} +\left(1-\frac{p}{2}\right) \conj{G}\right]\,d\theta \\
&\le \frac{1}{\|\phi\|_n}\|\ns{f}_n\|_{H^p}
\left\| \left(\frac{p}{2}\right)g + \left(1 - \frac{p}{2}\right)G\right\|_{H^q} 
\end{split}
\ees
Minkowski's inequality now gives 
\bes
\| \ns{f}_n\|_{H^p}^p \le \frac{1}{\|\phi\|_n} \|\ns{f}_n\|_{H^p}
\left( \frac{p}{2} \|g\|_{H^q} + \left| 1 - \frac{p}{2} \right| \|G\|_{H^q}
\right).
\ees
Since $p-1 = \tfrac{p}{q},$ this shows that 
$$\left(\frac{1}{2\pi}\int_0^{2\pi} |\ns{f}_n(re^{i\theta})|^p 
\,d\theta\right)^{1/p} \le 
\frac{1}{\|\phi \|_n^{q/p}}
\left( \frac{p}{2} \|g\|_{H^q} + \left| 1 - \frac{p}{2} \right| \|G\|_{H^q}
\right)^{q/p} , 
$$  when $0<r<1.$ 
From Theorem \ref{subspace_approx} and the fact 
that convergence in $A^p$ implies uniform convergence on compact 
subsets of the disc, it follows that 
$$\left(\frac{1}{2\pi}\int_0^{2\pi} |\ns{f}(re^{i\theta})|^p 
\,d\theta\right)^{1/p} \le 
\frac{1}{\|\phi \|^{q/p}}
\left( \frac{p}{2} \|g\|_{H^q} + \left| 1 - \frac{p}{2} \right| \|G\|_{H^q}
\right)^{q/p} , $$ 
when $0<r<1.$ 
Now we apply \eqref{Omega_estimate} to infer that 
$$
\|\ns{f}\|_{H^p} \le \Bigg\{ \bigg(\frac{p}{2} + \bigg|1 - \frac{p}{2}\bigg|
\bigg)
\frac{\|g\|_{H^q}}{ \| \phi \|}\Bigg\}^{q/p}.$$ 
Since we know that $\| g \|_{A^q} \le C_p \| \phi \|$, and that 
$p/2 + |1 - (p/2)| = \max(p-1, 1)$, we conclude finally 
that the inequality \eqref{ryabkh_estimate} holds
under the assumption that $g\in C^1(\overline{\D}).$

If $g$ is a general function in $H^q,$ 
we can approximate 
it in $H^q$ norm by a sequence of functions 
$g_n \in C^1(\overline{\D})$.  (We may even use polynomials, by 
\cite{D_Hp}, Theorem 3.3.) 
Then the corresponding functionals $\phi_n$ converge to $\phi$, so 
by Theorem \ref{solution_continuity}, the extremal elements $\ns{f_n}$ 
for $\phi_n$ converge to the extremal element $\ns{f}$ for $\phi$  
in $A^p$ norm. 
Since $g_n \in C^1(\overline{\D})$, we have already found that 
$$\left(\frac{1}{2\pi}\int_0^{2\pi} |\ns{f}_n(re^{i\theta})|^p 
\,d\theta\right)^{1/p}
 \le \Bigg\{  \bigg[ \max(p-1,1)
\bigg]
\frac{C_p \|g_n\|_{H^q}}{\| g_n \|_{A^q} }\Bigg\}^{q/p}, \qquad 0<r<1. 
$$
But the convergence of $\ns{f_n} \rightarrow \ns{f}$ in $A^p$ norm
implies that  
$\ns{f_n}(z) \rightarrow \ns{f}(z)$ locally uniformly, so it follows that 
$$\left(\frac{1}{2\pi}\int_0^{2\pi} |\ns{f}(re^{i\theta})|^p 
\,d\theta\right)^{1/p}
 \le \Bigg\{
 \bigg[ \max(p-1,1)\bigg]
\frac{C_p \|g\|_{H^q}}{\| g \|_{A^q} }\Bigg\}^{q/p}, \qquad 0<r<1,
$$
which proves \eqref{ryabkh_estimate}. 
\end{proof}

\bibliographystyle{amsplain}

\end{document}